\documentclass[reqno]{amsart}

\usepackage{amscd,amssymb,amsmath,graphicx,verbatim}
\newtheorem {thm}{Theorem}[section]
\newtheorem{lem}[thm]{Lemma}
\newtheorem{prop}[thm]{Proposition}
\newtheorem{cor}[thm]{Corollary}
\newtheorem{df}[thm]{Definition}

\newtheorem{ex}[thm]{Example}
\newtheorem{exs}[thm]{Examples}

\newtheorem{quo}[thm]{Question}

\begin{document}

\title[$e$-Reduced rings in terms of the Zhou radical]{$e$-Reduced rings in terms of the Zhou radical}

\author{Handan Kose}
\address{Handan Kose, Department of Computer Science, Ankara
University, Ankara, Turkey} \email{handankose@ankara.edu.tr}

\author{Burcu Ungor}
\address{Burcu Ungor, Department of Mathematics, Ankara
University, Ankara, Turkey}\email{bungor@science.ankara.edu.tr}

\author{Abdullah Harmanci}
\address{Abdullah Harmanci, Department of Mathematics, Hacettepe
University, Ankara,~ Turkey}\email{harmanci@hacettepe.edu.tr}

\date{}

\begin{abstract} Let $R$ be a ring, $e$ an idempotent of $R$
and $\delta(R)$  denote the intersection of all essential maximal
right ideals of $R$ which is called Zhou radical. In this paper,
the Zhou radical of a ring is applied to the $e$-reduced property
of rings. We call the ring $R$ {\it Zhou right} (resp. {\it left})
{\it $e$-reduced}  if for any nilpotent $a$ in $R$, we have $ae\in
\delta(R)$ (resp. $ea\in \delta(R))$. Obviously, every ring is
Zhou $0$-reduced and a ring $R$ is Zhou right (resp., left)
$1$-reduced if and only if $N(R)\subseteq \delta(R)$. So we assume
that the idempotent $e$ is nonzero. We investigate basic
properties of  Zhou right $e$-reduced rings. Furthermore, we
supply some sources of examples for Zhou right $e$-reduced rings.
In this direction, we show that right $e$-semicommutative rings
(and so right $e$-reduced rings and $e$-symmetric rings), central
semicommutative rings and weak symmetric rings are Zhou right
$e$-reduced.  As an application, we deal with some extensions of
Zhou right $e$-reduced rings. Full matrix rings need not be Zhou
right $e$-reduced, but we present some Zhou right $e$-reduced
subrings of full matrix rings over
Zhou right $e$-reduced rings.\\[2mm]

\end{abstract}
\maketitle
\section{Introduction}
Throughout this paper, all rings are associative with identity.
For a ring $R$, we use $N(R)$, Id$(R)$, $U(R)$ and $C(R)$ to
represent the set of all nilpotents, the set of all idempotents,
the set of all invertible elements and the center of $R$,
respectively. Also, $J(R)$ and $\delta(R)$ stand for the Jacobson
radical and the Zhou radical of a ring $R$, respectively. Denote
the $n\times n$ full (resp., upper triangular) matrix ring over
$R$ by $M_n(R)$ (resp., $U_n(R))$, and $D_n(R)$ denotes the
subring of $U_n(R)$ having all diagonal entries are equal and
$V_n(R) = \{A = (a_{ij})\in D_n(R)\mid a_{ij} = a_{(i+1)(j+1)}$
for $i = 1,\dots,n - 2$ and $j = 2,\dots,n-1\}$ is a subring of
$D_n(R)$. Also, $\Bbb Z$ and $\Bbb Z_n$ denote the ring of
integers and the ring of integers modulo $n$.

The notion of reduced ring and its various generalizations have
been comprehensively studied in the literature. A ring is called
{\it reduced} if it has no nonzero nilpotent elements. Reduced
rings are extended to the $e$-reduced rings in \cite{MW} and
central reduced rings in \cite{UHKH}. Quasi-reduced rings, a
weaker condition than being central reduced is defined in
\cite{KUHH}. Let $R$ be a ring and $e\in $ Id$(R)$. Then $R$ is
called {\it left} (or {\it right}) {\it $e$-reduced} if $eN(R) =
0$ (or $N(R)e = 0$), and $R$ is said to be {\it central reduced}
if every nilpotent element of $R$ is central. In \cite{KUHH}, a
ring $R$ is called {\it quasi-reduced} if for any $a, b\in R$, $ab
= 0$ implies $(aR)\cap (Rb)$ is contained in the center of $R$. As
an another generalization of the reducedness, in \cite{CGHH}, a
ring $R$ is called {\it $J$-reduced} if $N(R)\subseteq J(R)$.

A weaker condition than ``reduced" was defined by Lambek in
\cite{L}, that is, a ring $R$ is called {\it symmetric} if having
$abc =0$ implies $acb = 0$ for all $a, b, c\in R$. Symmetric rings
are generalized to weakly symmetric rings in \cite{UKYH}, weak
symmetric rings in \cite{OC} and $e$-symmetric rings in \cite{MW}.
A ring $R$ is said to be {\it weakly symmetric} if for all $a, b,
c, r\in R$, $abc\in N(R)$ implies  $Racrb\subseteq N(R)$,
equivalently, $abc\in N(R)$ implies  $acrbR\subseteq N(R)$. A ring
$R$ is called {\it weak symmetric} if $abc\in N(R)$ implies
$acb\in N(R)$ for all $a, b, c\in R$. A ring $R$ with $e\in $
Id$(R)$ is called {\it $e$-symmetric} if $abc = 0$ implies $acbe =
0$ for all $a, b, c\in R$. It is known that right $e$-reduced
rings are $e$-symmetric. Also, in \cite{MWC}, a ring $R$ with
$e\in $ Id$(R)$ is said to be {\it weak $e$-symmetric} if $abc =
0$ implies $eacbe = 0$ for all $a, b, c\in R$.

As a generalization of symmetricity, in \cite{Sh},
semicommutativity of a ring is defined as follows: a ring $R$ is
called {\it semicommutative} if for any $a, b\in R$, $ab = 0$
implies $aRb = 0$. There are many papers to investigate
semicommutative rings and their generalizations. It is well known
that every semicommutative ring is abelian. The notion of central
semicommutative ring is introduced in \cite{OAH}. A ring $R$ is
called {\it central semicommutative} if for any $a, b\in R$, $ab =
0$ implies $aRb\subseteq C(R)$. Another generalization of
semicommutativity is $e$-semicommutativity which is defined in
\cite{KUH}. A ring $R$ is called {\it right}  (resp. {\it left})
{\it $e$-semicommutative} if for any $a, b\in R$, $ab = 0$ implies
$aRbe = 0$ (resp. $eaRb = 0$). The ring $R$ is called {\it
$e$-semicommutative} in case $R$ is both right and left
$e$-semicommutative. Also, in \cite{X}, a ring $R$ is said to be
{\it $J$-semicommutative} if for any $a, b \in R$, $ab = 0$
implies $aRb\subseteq J(R)$.

Zhou introduced the notion of $\delta$-small submodule in
\cite{Zh}. Let $M$ be a module and $N$ be a submodule of $M$. Then
$N$ is called {\it $\delta$-small} in $M$ if whenever $M = N + K$
and $M/K$ is singular where $K$ is a submodule of $M$, then $M =
K$. The sum of $\delta$-small submodules is denoted by
$\delta(M)$. Considering the ring $R$ as a right $R$-module over
itself, the ideal $\delta(R)$ is introduced as a sum of
$\delta$-small right ideals of $R$. By \cite[Corollary 1.7]{Zh},
$J(R/$Soc$(R_R))=\delta(R)/$Soc$(R_R)$. Since Zhou introduced the
delta submodule, $\delta(M)$ is named the {\it Zhou radical} of
$M$.

In ring theory, the Zhou radical and the notion of reducedness and
related notions play important roles and have generated wide
interest. With this motivation, in this paper, we combine these
concepts by investigating the $e$-reducedness within the Zhou
radical, which we refer to as ``Zhou $e$-reducedness". We
summarize the contents of the paper. In Section 2, we continue to
investigate some properties of the Zhou radical to use in the
sequel of the paper. In Section 3, we focus on the Zhou
$e$-reduced rings. Some examples are supplied to show that Zhou
$e$-reduced rings are abundant. In Section 4, we deal with some
extensions of Zhou $e$-reduced rings. Finally, in Section 5, we
investigate some matrix rings in terms of the Zhou
$e$-reducedness.

\section{Some properties of the Zhou radical}
We begin with the equivalent conditions for $\delta(R)$ of  a ring
$R$ and the submodule $\delta(M)$ of a module $M$ which are
mentioned and proved in \cite{Zh}. The Zhou radical plays a
crucial role as a tool in studying the structure of some classes
of rings and modules. In \cite{HKB}, duo property of rings
investigated by using properties of the Zhou radical. In this
section, some results are mentioned and proved to use in the
sequel of the paper.  Lemma \ref{besli} and Lemma \ref{mod} are
studied in \cite{Zh}.

\begin{lem}\label{besli} Given a ring $R$, each of the following sets is equal to $\delta(R)$.\begin{enumerate} \item[(1)] $R_1$ = the intersection
of all essential maximal right ideals of $R$.
\item[(2)] $R_2$ = the unique largest $\delta$-small right ideal of $R$.
\item[(3)] $R_3 = \{x\in R\mid xR + K_R = R$ implies $K_R$ is a direct summand of $R_R\}$.
\item[(4)] $R_4 = \bigcap\{$ideals $P$ of $R\mid R/P$ has a faithful singular simple module$\}$.
\item[(5)] $R_5 = \{x\in R\mid$ for all $y\in R$, there exists a semisimple right ideal $Y$ of $R$ such that $(1 + xy)R\oplus Y = R_R\}$.
\end{enumerate}
\end{lem}
\begin{lem}\label{mod} Let $R$ be a ring. Then the following hold.
\begin{enumerate}
\item $(eRe)\cap \delta(R) = \delta(eRe) = e\delta(R)e$ for any $e\in $ Id$(R)$.
\item Let $\{M_i\}_{i\in I}$ be a family of $R$-modules. Then $\delta(\oplus_{i\in I} M_i) = \oplus_{i\in I}\delta(M_i)$.
\item Let $M$ and $N$ be $R$-modules. If $f \colon M\rightarrow N$ is a homomorphism, then $f(\delta(M))\subseteq \delta(N)$.
\end{enumerate}
\end{lem}

We start with some examples of the Zhou radical  and nilpotents of
some rings.

\begin{ex}\label{ilk}{\em Let $R$ be a ring and $n$ a positive
integer.
\begin{enumerate}
\item[(1)]
$\delta(U_2(R)) =
\delta\left(\begin{bmatrix}R&R\\0&0\end{bmatrix}\right)+\delta
\left(
\begin{bmatrix}0&0\\0&R\end{bmatrix}\right)=
\delta\left(\begin{bmatrix}R&R\\0&0\end{bmatrix}\right)+\begin{bmatrix}0&0\\0&\delta(R)\end{bmatrix}$,
$$\delta(U_3(R)) =
\delta\left(\begin{bmatrix}R&R&R\\0&0&0\\0&0&0\end{bmatrix}\right)
+
\delta\left(\begin{bmatrix}0&0&0\\0&R&R\\0&0&0\end{bmatrix}\right)
+\begin{bmatrix}0&0&0\\0&0&0\\0&0&\delta(R)\end{bmatrix}$$
\noindent \hspace*{-0.8cm} and $\delta(M_n(R)) = M_n(\delta(R))$.
\item[(2)]  $N(U_2(R)) = \begin{bmatrix}N(R)&R\\0&N(R)\end{bmatrix}$ and $N(U_3(R)) = \begin{bmatrix}N(R)&R&R\\0&N(R)&R\\0&0&N(R)\end{bmatrix}$.
\end{enumerate}
Let $D$ be a division ring. Then we have the following.
\begin{enumerate}
\item[(3)] $\delta(U_2(D)) = \begin{bmatrix}0&D\\0&D\end{bmatrix}$ and $\delta(U_3(D)) =
\begin{bmatrix}0&D&D\\0&0&D\\0&0&D\end{bmatrix}$.
\item[(4)] $N(U_2(D)) = \begin{bmatrix}0&D\\0&0\end{bmatrix}$ and
 $N(U_3(D)) = \begin{bmatrix}0&D&D\\0&0&D\\0&0&0\end{bmatrix}$.
\end{enumerate}}
\end{ex}

\begin{ex}\label{orn}\em Let $F$ be a field and $A = \mathbb Z_2<a, b>$ the free algebra with noncommuting indeterminates $a$, $b$ over $F$. Let $I$ be
the ideal of $A$ generated by $aAb$, $a^2 - a$ and $b^2 - b$.
Consider the ring $R = A/I$ and identify the elements in $A$ with
their images in $R$ for simplicity. It is easily checked that
\begin{center} $R = \{0, 1, a, b, ba, a + b, a + ba, b + ba, a + b + ba, 1
+ a, 1 + b, 1 + ba, 1 + a + b,\linebreak 1 + a + ba, 1 + b + ba, 1
+ a + b + ba\}$.\end{center} Then $aR =\{0, a\}$, $(ba)R = \{0,
ba\}$, $(1 + a + b + ba)R = \{0, 1 + a + b + ba\}$ and $(a + ba)R
= \{0, a + ba\}$ are minimal right ideals of $R$. It follows that
$Soc(R) = aR\oplus (ba)R\oplus (1 + a + b + ba)R\oplus (a + ba)R$,
and $Soc(R)= \delta(R)$ is the Zhou radical of $R$. Next we
determine the Zhou radical $\delta(U_2(R))$ of $U_2(R)$. It is
obvious that $\begin{bmatrix} Soc(R)&R\\0&0\end{bmatrix}$ is the
unique maximal essential right ideal in the first row. So the Zhou
radical of the first row is
$\begin{bmatrix}Soc(R)&R\\0&0\end{bmatrix}$. Then $\delta(U_2(R))
=
\begin{bmatrix}\delta(R)&R\\0&\delta(R)\end{bmatrix}$.
\end{ex}

Note that $J(R/J(R)) = 0$ for a ring $R$. One may suspect whether
$\delta(R/\delta(R)) = 0$. But there are rings which erase the
possibility as shown below.

\begin{ex}\rm Let $F$ be a field and consider the ring $R = U_2(F)$.
By the preceding example, $R/\delta(R) \cong
\begin{bmatrix}F&0\\0&0\end{bmatrix}$. Hence $\delta(R/\delta(R))
\cong \delta\left(\begin{bmatrix}F&0\\0&0\end{bmatrix}\right)=
\begin{bmatrix}F&0\\0&0\end{bmatrix}$.
\end{ex}
\begin{prop}\label{semi} For any ring $R$, $\delta(R)$ is a semiprime ideal.
\end{prop}
\begin{proof} Let $a\in R$ and assume that $aRa\subseteq \delta(R)$ but $a\notin \delta(R)$. Then there exists an essential maximal right ideal $I$ of $R$ such that $a\notin I$. Then $aR + I = R$. So there exist $r\in R$ and $y\in I$ such that $1 = ar + y$. Then we get $a = ara + ya$. It yields $a\in I$ since $ara\in \delta(R)\subseteq I$ and $ya\in I$. This is a contradiction. So the result follows.
\end{proof}

The following result is used in the sequel. It is probably well
known. However, no reference is fixed to refer and so we give a
proof for the sake of completeness.

\begin{lem}\label{delI} Let $R$ be a ring and $I$ an ideal of $R$. Then $\delta(I)\subseteq I\cap
\delta(R)$. The reverse inclusion holds if $I$ is maximal.
\end{lem}
\begin{proof} By Lemma \ref{mod}(3), $\delta(I)\subseteq \delta(R)$.
Then $\delta(I)\subseteq I\cap \delta(R)$. For the reverse
inclusion, assume that $I$ is maximal. Consider the following
cases:\\
Case (1) If $I$ is essential, then $\delta(R)\subseteq I$. So $\delta(R)\subseteq \delta(I)$.\\
Case (2) If $I$ is not essential, then there exists an ideal $K$
of $R$ such that $I\oplus K = R$. This entails $\delta(I) \oplus
\delta(K) = \delta(R)$. In this case, $\delta(I) = I\cap
\delta(R)$.
\end{proof}

One can ask that for a ring $R$ and its essential ideal $I$ which
is not maximal, $I\cap \delta(R)\subseteq \delta(I)$ always hold,
but the next example shows that this inclusion need not be true in
general.

\begin{ex}\rm Consider the ring $R=\Bbb Z_{16}$ and its ideal $I=\overline{4}\Bbb
Z_{16}$. On the one hand, $I$ is essential but not maximal in $R$.
On the other hand, $\delta(R)=\overline{2}\Bbb Z_{16}$ and
$\delta(I)=\overline{8}\Bbb Z_{16}$. Then $I\cap
\delta(R)=\overline{4}\Bbb Z_{16}=I$, and so $I\cap \delta(R)$ is
not contained in $\delta(I)$.
\end{ex}

In \cite{Do}, Dorroh gave a way to embed a ring $R$ without an
identity into a ring with an identity $\Bbb Z\oplus R$, which is
called a Dorroh extension of $R$. In ring theory, Dorroh extension
has become an important method of constructing new rings and
analyzing properties of rings. Let $R$ be a ring and $T$ be an
associative ring that may not possess an identity and an $(R,
R)$-bimodule obeying multiplication in $T$, that is, for any $a\in
R$ and $t$, $s\in T$, $a(ts) = (at)s$, $t(as) = (ta)s$ and $(ts)a
= t(sa)$. The {\it Dorroh extension} (in other words, {\it ideal
extension}) of $T$ by $R$, denoted by $D(R, T)$, is the abelian
group $R\times T$ with multiplication defined by $(a_1, t_1)(a_2,
t_2) = (a_1a_2, a_1t_2 + t_1a_2 + t_1t_2)$ for $a_1$, $a_2\in R$
and $t_1$, $t_2\in T$. Note that $(1, 0)$ is the identity of $D(R,
T)$. Mesyan in \cite[Proposition 5]{Mes} characterized maximal
ideals and right (or left) ideals of Dorroh extensions. In this
case a map $\varphi \colon T\rightarrow R$ is said to be an
$R$-homomorphism provided it is a ring homomorphism that is also
an $(R, R)$-bimodule homomorphism. Maximal right ideals and the
Zhou radical $\delta(D(R, T))$ are characterized as in the
following. See \cite[Proposition 5]{Mes} for details.
\begin{lem}\label{IR} Let $T$ be an algebra over a ring $R$. Consider
the Dorroh extension $D(R, T)$ of $T$ by $R$ and let $K = \{(a,
-t) \mid a\in R, t\in T, ~ a -\varphi(t)\in Z\}$ be a maximal
right ideal in $D(R, T)$, where $Z\subseteq A$ is a maximal right
ideal of $R$, $J$ is an $R$-subring of $T$, and $\varphi:
J\rightarrow A/Z$ is a surjective $R$-homomorphism such that for
all $(a, -j)\in K$ and $i\in T$ the following are satisfied
\begin{enumerate}
\item[(a)] $ai - ji \in$ ker$(\varphi)$,
\item[(b)] $ia - ij\in$ ker$(\varphi)$.
\end{enumerate}
Then we have the following.
\begin{enumerate}
\item [(1)] If $\varphi(T)\subseteq Z$, then $K$ has the form $K = A\oplus T$ for some maximal right ideal $A$ of $R$.
\item [(2)] If $\varphi(T)\nsubseteqq Z$, then $K\subseteq A\oplus T$ for some maximal right ideal $A$ of $R$.
\item [(3)] $\delta(D(R, T)) = \delta(R)\oplus T$.
\item [(4)] Let $(a, t)\in D(R, T)$. Then $(a, t)\in$ Id$(D(R, T))$ if and only if $a\in$ Id$(R)$ and $(a + t)^2 = (a + t)$.
\item [(5)] Let $(a, t)\in D(R, T)$. Then $(a, t)\in N(D(R, T))$ with $(a, t)^n = 0$ if and only if
$a^n = 0$ and $(a + t)^n = 0$.
\end{enumerate}
\end{lem}
\begin{proof} (1), (2) and (3) are known by \cite[Lemma 2.8]{HKB}.\\
(4) Let $(a, t)\in D(R, T)$. Then $(a, t)\in$ Id$(D(R, T))$ if and
only if $a^2 = a$ and $at + ta + t^2 = t$ if and only if  $a^2 =
a$ and $(a + t)^2 = a + t$.\\(5) Let $(a, t)\in D(R, T)$. Then
$(a, t)\in N(D(R, T))$ if and only if $(a, t)^2 = 0$ if and only
if $a^2 = 0$ and $at + ta + t^2 = 0$ if and only if $a^2 = 0$ and
$(a + t)^2 = 0$. In this way we may continue to reach $(a, t)^n =
0$ if and only if $a^n = 0$ and $(a + t)^n = 0$ for each positive
integer $n$.
\end{proof}


We now illustrate the Zhou radical of a Dorroh extension with an
example.

\begin{ex}\label{Me} {\rm Consider the rings $R = U_2(F)$ and $T =
M_2(F)$ for a field $F$. It is obvious that  $\delta(R) =
\begin{bmatrix}0&F\\0&F\end{bmatrix}$ and $\delta(T) = T$. Also,
$\delta(D(R, T)) = \left(\begin{bmatrix}0&F\\0&F\end{bmatrix},
\begin{bmatrix}F&F\\F&F\end{bmatrix}\right)$.}
\end{ex}

Let $R$ be a ring and $S$ a multiplicatively closed subset of $R$
consisting of the identity 1 and some central regular elements,
that is, for any element $s\in S$ and $r\in R$, $sr = 0$ implies
that $r = 0$. Consider the ring $S^{-1}R = \{s^{-1}r \mid s\in S,
r\in R\}$.
\begin{lem}\label{del} Let $R$ be a ring and $S$ a multiplicatively closed subset of $R$ consisting of the identity 1 and some central regular elements. Then $S^{-1}\delta(R)\subseteq \delta(S^{-1}R)$.
\end{lem}
\begin{proof} The natural map $\varphi$ from $R$ to $S^{-1}R$ defined by $\varphi(r) = 1^{-1}r$ is a ring homomorphism and we may identify $\delta(R)$ with $\varphi(\delta(R))$. So we have $\delta(R)\subseteq \delta(S^{-1}R)$. It entails that $S^{-1}\delta(R)\subseteq \delta(S^{-1}R)$.
\end{proof}

 There are rings $R$ and $S$ such that the inclusion $\delta(S^{-1}R)\subseteq S^{-1}\delta(R)$ need not hold in general.
\begin{ex}\label{local-ex}\em Let $R$ denote the ring of integers $\mathbb Z$ and $S = R\setminus (0)$. Then $S^{-1}R = \mathbb Q$ is the rational numbers. It is well known that $\delta(\mathbb Z) = 0$ and $\delta(\mathbb Q) = \mathbb Q$.
\end{ex}

\section{Zhou $e$-reduced rings}

In the context, the Zhou radical and idempotents are used to
determine the structure of the rings.  In this section, we combine
the Zhou radical with an idempotent $e$ of the ring to define Zhou
right (resp., left) $e$-reduced ring as follows. We start with an
example for a motivation as follows.

\begin{ex}\label{ilk}{\em Let $R$ be a reduced ring. Then we have
the following.
\begin{enumerate}
\item $N(U_2(R)) = \begin{bmatrix}0&R\\0&0\end{bmatrix}$ and
$N(U_3(R)) = \begin{bmatrix}0&R&R\\0&0&R\\0&0&0&\end{bmatrix}$.
\item $EN(U_2(R))\subseteq \delta(U_2(R))$ and $N(U_2(R))E\subseteq \delta(U_2(R))$ for each $E\in$ Id$(U_2(R))$.
\item $EN(U_3(R))\subseteq \delta(U_3(R))$ and $N(U_3(R))E\subseteq \delta(U_3(R))$ for each $E\in$ Id$(U_3(R))$.
\end{enumerate}
    }\end{ex}

Motivated by Example \ref{ilk}, we give the main definition of
this paper.
\begin{df}\em Let $R$ be a ring and $e\in $ Id$(R)$. Then  $R$ is called {\it  Zhou right {\rm (resp.,} left{\rm )} $e$-reduced} provided
that $N(R)e\subseteq \delta(R)$ (resp., $eN(R)\subseteq
\delta(R))$. A ring $R$ is called {\it Zhou $e$-reduced} if it is
both Zhou right $e$-reduced and Zhou left $e$-reduced.
\end{df}

It is clear that every ring is Zhou right (resp., left)
$0$-reduced. Also, a ring $R$ is Zhou right (equivalently, left)
$1$-reduced if and only if $N(R)\subseteq \delta(R)$. In the
sequel, we assume that $e\in$ Id$(R)\setminus \{0\}$. Obviously,
every $e$-reduced ring, every semicommutative ring and every local
ring is Zhou right $e$-reduced. We now give some sources for Zhou
right $e$-reduced rings.

\begin{prop}\label{lem}\begin{enumerate}
\item[(1)] Every central semicommutative ring is Zhou $e$-reduced.
\item[(2)] Every right $e$-semicommutative ring is Zhou right $e$-reduced.
\item[(3)] Every Zhou right $1$-reduced ring is Zhou right $e$-reduced.
\item[(4)] Every semisimple ring is Zhou  $e$-reduced.
\item[(5)] Every weakly symmetric ring is Zhou  $e$-reduced.
\item[(6)] Every weak symmetric ring is Zhou  $e$-reduced.
\item[(7)] Every $J$-reduced ring is Zhou  $e$-reduced.
\end{enumerate}
\end{prop}
\begin{proof}
(1) Let $a^n = 0$ for some integer $n\geq 2$. Then
$a^{n-1}Ra\subseteq C(R)$. Commuting $a^{n-1}Ra$ with $Ra$, we get
$a^{n-1}(Ra)^2 = 0$. Since $R$ is central semicommutative,
$(a^{n-2}Ra)(Ra)^2\subseteq C(R)$. Commuting $a^{n-2}(Ra)^3$ with
$Ra$, we have $a^{n-2}(Ra)^4 = 0$. The ring $R$ being central
semicommutative implies $(a^{n-3}Ra)(Ra)^4\subseteq C(R)$.
Commuting $a^{n-3}(Ra)^5$ with $Ra$, we get $a^{n-4}(Ra)^6=0$.
Continuing in this way, it entails
 $a^{n-(n-2)}(Ra)^n = 0$. So $a^2(Ra)^n = 0$. As $R$ is
central semicommutative, $aRa(Ra)^n\subseteq C(R)$. Commuting
$aRa(Ra)^n$ with $Ra$, we get $(Ra)^{n+3} = 0$. So $Ra\subseteq
J(R)$. It follows that $ae\in \delta(R)$. Hence $R$ is  Zhou right
$e$-reduced.\\
(2) Assume  that $R$ is right $e$-semicommutative. Firstly, note
that $e$ is left semicentral, that is, $re = ere$ since $(1-e)re =
0$ for each $r\in R$. Let $a\in N(R)$ with nilpotency index $n$
for some $n\geq 2$. Since $R$ is right $e$-semicommutative,
$a^{n-1}Rae = 0$.  Continuing in this way we have
$a^{n-2}(Rae)(Rae) = 0$. By induction, $(Rae)^n = 0$. Since
$\delta(R)$ contains right or left nilpotent ideals,
$Rae\subseteq \delta(R)$. So $ae\in\delta(R)$.\\
(3) is clear by $N(R)\subseteq \delta(R)$ and  (4) is clear by the fact that if $R$ is semisimple, then $\delta(R) = R$.\\
(5) Let $R$ be a weakly symmetric ring and $a\in N(R)$. By
\cite[Theorem 2.17]{UKYH}, $Rara$ is a nil left ideal for each
$r\in R$. Then $Rara\subseteq \delta(R)$. Hence $aRa\subseteq
\delta(R)$. By Proposition \ref{semi},
$a\in \delta(R)$. Thus $ae, ea\in \delta(R)$.\\
(6) Let $a\in N(R)$. By \cite[Theorem 2.2]{HKU}, $ra\in N(R)$ and also $ar\in N(R)$ for each $r\in R$.
Since  nil right ideals and  nil left ideals are contained in $\delta(R)$, we have $ae, ea\in \delta(R)$.\\
(7) Let $a\in N(R)$. Since $R$ is $J$-reduced, $J(R)$ is an ideal
in $R$ and $J(R)\subseteq \delta(R)$, we have $ae, ea\in
\delta(R)$. So $R$ is Zhou $e$-reduced.
\end{proof}

We may produce many examples by Proposition \ref{lem} as follows.

\begin{exs}\label{Orn}{\em The following hold.
\begin{enumerate}
\item[(1)] Let $F$ be a field. Then $M_n(F)$ is Zhou right  $e$-reduced,
but neither central semicommutative nor $e$-semicommutative for
some $e\in$ Id$(M_n(F))$.
\item[(2)] Let $R$ be a reduced ring. Then the rings $U_n(R)$, $D_n(R)$ and $V_n(R)$ are Zhou right  $e$-reduced for any $n\in \mathbb
N$.
\end{enumerate}}
\end{exs}

\begin{proof} (1) For any field $F$, $M_n(F)$ is a semisimple ring. So it is clear since $\delta(M_n(F)) = M_n(F)$.\\
(2) For a reduced ring $R$, it is clear that $N(U_n(R))\subseteq
\delta(U_n(R))$, $N(D_n(R))\subseteq \delta(D_n(R))$ and
$N(V_n(R))\subseteq \delta(V_n(R))$. By Proposition \ref{lem}(3)
and make using the facts $N(U_n(R))$, $N(D_n(R))$ and $N(V_n(R))$
are ideals in $\delta(U_n(R))$, $\delta(D_n(R))$ and
$\delta(V_n(R))$, respectively, the result follows.
\end{proof}


\begin{ex} \em Let $R$ denote the ring in Example \ref{orn}. Since $ba\in R$ is the unique nonzero nilpotent in $R$,
nilpotent elements of $U_2(R)$ are of the forms $X =
\begin{bmatrix}ba&*\\0&0\end{bmatrix}$, $Y =
\begin{bmatrix}ba&*\\0&ba\end{bmatrix}$,  $Z =
\begin{bmatrix}0&*\\0&ba\end{bmatrix}$ and $T=\begin{bmatrix}0&*\\0&0\end{bmatrix}$. Since $\{X, Y,
Z,T\}\subseteq \delta(U_2(R))$, $U_2(R)$ is  Zhou right
$E$-reduced  and Zhou left $E$-reduced  for each $E\in$
Id$(U_2(R))$.
\end{ex}

Recall that a ring $R$ is called {\it right {\rm (}quasi-{\rm
)}duo} if every (maximal) right ideal of $R$ is two-sided. A {\it
left {\rm(}quasi-{\rm)}duo ring} is defined analogously. A ring is
said to be {\it {\rm(}quasi-{\rm)}duo} if it is both right
(quasi-)duo and left (quasi-)duo. One might think that
$N(R)\subseteq \delta(R)$, i.e., $R$ is Zhou right $1$-reduced.
But this is not the case in general. In the following, we show
that this containment is true for quasi-duo rings.

\begin{lem}\label{div} Every simple quasi-duo ring is a division ring.
\end{lem}
\begin{proof} Let $0\neq a\in R$. Consider the left ideal $Ra$. Since $R$ is quasi-duo,
$Ra$ is an ideal. Hence $Ra = R$. Thus $a$ is left invertible, and
similarly, it is also right invertible. Therefore $R$ is a
division ring.
\end{proof}
\begin{prop}\label{Red} Let $R$ be a quasi-duo ring. Then $R/\delta(R)$ is a reduced ring.
\end{prop}
\begin{proof} Suppose that  $R$ is a quasi-duo ring. By definition, $\delta(R)$ is intersection of maximal essential
right ideals $\{M_i\}_{i\in I}$ of $R$. For each $i\in I$, the
canonical map $R\rightarrow R/I_i$ induces an injection $\alpha
\colon R/\delta(R)\rightarrow \prod_{i\in I} R/M_i$. As $R$ is
quasi-duo, for each $i\in I$, all $M_i$ are ideals. Hence $R/M_i$
are all simple rings. By Lemma \ref{div}, $R/M_i$ are all division
rings. It follows that $R/\delta(R)$ is reduced.
\end{proof}
\begin{cor} If $R$ is a quasi-duo ring and $\delta(R) = 0$, then $R$ is reduced.
\end{cor}
\begin{thm}\label{NR} Every right quasi-duo ring is Zhou right $1$-reduced.
\end{thm}
\begin{proof} For the sake of completeness, we imitate the proof of \cite[Lemma 2.3]{Yu} to get the result.
Suppose that $R$ is right quasi-duo. Let $a\in R$ with $a^n = 0$
and $a^{n-1}\neq 0$ for some integer $n>1$. By a contradiction,
assume that $a\notin \delta(R)$. There exists a maximal essential
right ideal $M$ such that $aR + M = R$. Multiplying the latter by
$a$ from the left, we get $a^2R + aM = aR$. It entails that $a^2R
+ aM + M = R$. We continue multiplying the latter by $a$ from the
left, we get $a^{n-1}R + a^{n-2}M+\dots + aM + M = R$. At the
$n^{th}$-step we get $a^{n-1}M + a^{n-2}M +\dots + aM + M = R$.
Since $R$ is right quasi-duo, $a^iM\subseteq M$ where $1\leq i\leq
n-1$. Thus $M = R$. This is the required contradiction. Therefore
$N(R)\subseteq \delta(R)$. This means that $R$ is Zhou right
$1$-reduced.
\end{proof}

Immediately, we obtain the next result by Proposition \ref{lem}(3)
and Theorem \ref{NR}.

\begin{cor} Every right quasi-duo ring is Zhou right $e$-reduced.
\end{cor}

There are Zhou right  $e$-reduced rings which are not right
(quasi-)duo.

\begin{exs}\label{exg}\em (1) Let $F$ be a field and
$R = U_2(F)$. Then $\delta(R) =
\begin{bmatrix}0&F\\0&F\end{bmatrix}$. Let $I =
\left\{\begin{bmatrix}0&a\\0&a\end{bmatrix}\mid a\in F\right\}$
and $L = \left\{\begin{bmatrix}a&a\\0&0\end{bmatrix}\mid a\in
F\right\}$. Then $I$ is a right ideal but not left and $L$ is a
left ideal but not right. However, since $N(R)\subseteq
\delta(R)$, by Proposition \ref{lem} (3), $R$ is Zhou $e$-reduced
for $e\in$ Id$(R)$.\\
(2) For  a division ring $D$ and a positive integer $n\geq 2$, the
ring $M_n(D)$ is Zhou right $e$-reduced but not quasi-duo.
\end{exs}

We now give a characterization of Zhou right $e$-reduced rings by
some subring of direct product of rings.

\begin{lem} Let $R$ be a ring and $S = \{(r,s)\in R\times R\mid r - s\in
\delta(R)\}$. Then $\delta(S)=\{(r,s)\in \delta(R)\times
\delta(R)\mid r - s\in \delta(R)\}$.
\end{lem}
\begin{thm} A ring $R$ is Zhou right $e$-reduced if and only if the ring
$S = \{(r,s)\in R\times R\mid r - s\in \delta(R)\}$ is Zhou right
$(e, e)$-reduced.
\end{thm}
\begin{proof} $\Rightarrow:$ Assume that $R$ is a Zhou right $e$-reduced ring and $(r, s)\in N(S)$.
Then $r\in N(R)$ and $s\in N(R)$. By assumption, $re\in\delta(R)$,
$se\in \delta(R)$. Then $(r, s)(e, e) = (re, se)\in
\delta(R)\times \delta(R)$. Since $r - s\in \delta(R)$ and
$\delta(R)$
is an ideal in $R$, $(r - s)e = re - se\in\delta(R)$. So $S$ is  Zhou right $(e, e)$-reduced. \\
$\Leftarrow:$ Suppose that $S$ is  Zhou right $(e, e)$-reduced.
Let $r\in N(R)$. Then $(r, r)\in S$, in particular $(r, r)\in
N(S)$. By supposition, $(r, r)(e, e)\in \delta(R)\times
\delta(R)$. Since $(r, r)(e, e) = (re, re)$ and $(r, r)(e, e)\in
\delta(R)\times \delta(R)$, $re\in\delta(R)$. It follows that $R$
is Zhou right $e$-reduced.
\end{proof}

Under some restricted conditions on the ring, homomorphic images of Zhou right $e$-reduced rings are Zhou right $e$-reduced rings.
\begin{prop}\label{nil1} Let $R$ be a ring. Then the following hold.
\begin{enumerate}
\item[(1)] Let $I$ be an ideal of $R$  and $e^2 = e\in
I$. Assume that $\delta(I) = I\cap \delta(R)$. If $R$ is Zhou
right $e$-reduced, then so is $I$ as a ring without identity.
\item[(2)] Let $I$ be a nil ideal of  $R$. If $R$ is  Zhou
right $e$-reduced, then $R/I$ is Zhou right $e+I$-reduced.
\end{enumerate}
\end{prop}
\begin{proof} (1) Let $a\in I$ and assume that $a^n = 0$ for some $n > 1$. Since $R$ is Zhou right $e$-reduced,
$ae\in \delta(R)$. The assumption $\delta(I) = I\cap \delta(R)$ implies that $ae\in \delta(I)$.\\
(2) Let $a\in R$ with $a + I\in N(R/I)$. Then $a^n\in I$ for some
positive integer $n$. Since $I$ is nil, there exists a positive
integer $m$ such that $a^{nm} = 0$. The ring $R$ being Zhou right
$e$-reduced implies $ae\in \delta(R)$. Let $\pi\colon R\rightarrow
R/I$ denote the natural homomorphism with $\pi(r) = r + I$. Then
$\pi(ae) = ae + I$. Since $\pi(\delta(R))\subseteq \delta(R/I)$ by
Lemma \ref{mod}, we have $ae + I\in \delta(R/I)$.
\end{proof}

\begin{thm}  Let $\{R_i\}_{i\in I}$ be a family of rings where $I=\{1,2,\dots,n\}$ and
$R = \prod^n\limits_{i=1} R_i$ and $e_i^2 = e_i\in R_i$ for each
$i\in I$
 and set $e = (e_i)\in R$. Then $R_i$ is Zhou right $e_i$-reduced for each
 $i \in I$ if and only if $R$ is Zhou right $e$-reduced.
\end{thm}
\begin{proof} Note that $\delta(R) = \prod^n\limits_{i=1} \delta(R_i)$ by Lemma \ref{mod}(2).
Assume that $R_i$ is Zhou right $e_i$-reduced for each $i\in I$.
Let $a = (a_i)\in N(R)$.
  Then $a_i\in$ $N(R_i)$ for each $i\in I$. By assumption, $a_ie_i\in \delta(R_i)$ for each $i\in I$.
 Hence $ae\in \delta(R)$. Conversely, suppose that $R$ is Zhou right $e$-reduced.
Let $a_i\in N(R_i)$. Define $a = (x_i)\in R$ by $x_i = a_i$ and
$x_j = 0$ in case $i\neq j$. Then $a\in N(R)$. By supposition,
$ae\in \delta(R)$. It entails that $a_ie_i\in \delta(R_i)$.
\end{proof}

We close this section by observing some results about  corner
rings.

\begin{prop} Let $R$ be a Zhou right $e$-reduced ring. Then $eRe$
is Zhou right $f$-reduced for every $f\in $ Id$(eRe)$.
\end{prop}
\begin{proof} First we claim that $eRe$ is Zhou  right
$e$-reduced. Let $a\in N(eRe)$. Then $a\in N(R)$, and so $ae\in
\delta(R)$. Hence $ae=eae\in e\delta(R)e=\delta(eRe)$ by Lemma
\ref{mod}(1). Thus $N(eRe)e\subseteq \delta(eRe)$, as claimed.
Therefore $eRe$ is Zhou right $f$-reduced for every $f\in $
Id$(eRe)$ by Proposition \ref{lem}(3).
\end{proof}

\begin{prop} Let $R$ be a Zhou right $e$-reduced ring and $f\in$ Id$(R)$. If $e\in$ Id$(fRf)$, then $fRf$ is Zhou right
$e$-reduced.
\end{prop}
\begin{proof} Let $f\in$ Id$(R)$, $e\in$ Id$(fRf)$ and $a\in N(fRf)$. Since
$R$ is Zhou right $e$-reduced for $e\in$ Id$(fRf)$, $ae\in
\delta(R)$. Again by Lemma \ref{mod}(1), $\delta(fRf) =
f\delta(R)f$. Then $ae\in \delta(fRf)$. This completes the proof.
\end{proof}

\begin{quo}\em Let $f\in$ Id$(R)$, $e_1\in$ Id$(fRf)$, $e_2\in$ Id$(1-f)R(1-f)$. If $fRf$ is Zhou right $e_1$-reduced
and $(1-f)R(1-f)$ is Zhou right $e_2$-reduced, then is $R$ a Zhou
right $e$-reduced  ring for some $e\in$ Id$(R)$?
\end{quo}
A negative answer exists as the following example shows.
\begin{ex}\em Consider the ring $R = M_2(\mathbb Z_4)$ with
$e = \begin{bmatrix}0&0\\3&1\end{bmatrix}\in$ Id$(R)$ and $a =
\begin{bmatrix}0&1\\0&0\end{bmatrix}\in N(R)$. Let $f=\begin{bmatrix}1&0\\0&0\end{bmatrix}$. Then
$fRf\cong \mathbb Z_4$ and $(1-f)R(1-f)\cong \mathbb Z_4$ are Zhou
right $g$-reduced for each $g\in$ Id($\mathbb Z_4)$ since
$N(\mathbb Z_4) = \delta(\mathbb Z_4)$. Unfortunately,
$ae\notin\delta(R)$.
\end{ex}

\section{Some extensions of Zhou $e$-reduced rings}

In this section, we study some extensions of rings in terms of the
Zhou $e$-reduced property. Let $R$ be a ring and $S$ a
multiplicatively closed subset of $R$ consisting of the identity 1
and some central regular elements, that is, for any element $s\in
S$ and $r\in R$, $sr = 0$ implies that $r = 0$. Consider the ring
$S^{-1}R = \{s^{-1}r \mid s\in S, r\in R\}$. We suppose that every
idempotent in $S^{-1}R$ is of the form $s^{-1}e$ where $s\in S$
and $e\in$ Id$(R)$.
\begin{prop}\label{local} Let $S$ be a multiplicatively closed subset of a ring $R$ consisting of central regular elements
and $e\in$ Id$(R)$. If $R$ is Zhou right $e$-reduced and
$s^{-1}e\in$ Id$(S^{-1}R)$,  then $S^{-1}R$ is Zhou right
$s^{-1}e$-reduced.
\end{prop}
\begin{proof}  Assume that $R$ is Zhou right $e$-reduced. Let $t^{-1}a\in N(S^{-1}R)$. Then
$a\in N(R)$. By assumption, $ae\in \delta(R)$.
By construction,
$(t^{-1}a)(s^{-1}e) = (ts)^{-1}(ae)\in S^{-1}\delta(R)$. By Lemma
\ref{del},
$(ts)^{-1}(ae)\in \delta(S^{-1}R)$. Thus $S^{-1}R$ is Zhou right $s^{-1}e$-reduced.
\end{proof}

The following example is stated in \cite[Page 1967]{KL} connection
with the converse of Proposition \ref{local}. It is not true that
every element of Id$(S^{-1}R)$ has the form $s^{-1}e$ for some
$e\in$ Id$(R)$.
\begin{ex}\em Let $F$ be a field and $I$ the ideal generated by $x^2 - xy$ in $F[x, y]$.
Consider the ring $R = F[x, y]/I$. We denote the elements of $R$
without bar sign. Let $S = \{y^n\in R\mid n\geq 1\}$. Then
$(y^{-1}x)(y^{-1}x) = y^{-1}x\in$ Id$(S^{-1}R)$ but $x\notin$
Id$(R)$.
\end{ex}

\begin{thm}\label{groupring} Let  $G$ be a finite group and $F$ be a field. If the characteristic of $F$ does not divide the order of $G$, then the group ring $FG$ is
Zhou right $e$-reduced.
\end{thm}
\begin{proof}  The group ring $FG$ is semisimple by Maschke's Theorem. So $\delta(FG) =
FG$. This completes the proof.
\end{proof}

In Theorem \ref{groupring}, $F$ being a field is not superfluous
as shown below.

\begin{ex}\rm Let $G$ be a group. By \cite[Proposition 2.11(2)]{HKB},
$\delta(\mathbb ZG) = 0$. Since the group ring $\mathbb ZG$ for
any group $G$ may have no nontrivial idempotent elements, $\mathbb
ZG$ need not be Zhou right $e$-reduced.
\end{ex}

\begin{thm} Let $R$ be a ring, $T$ be a subring of $R$ not necessarily having an identity,
$e\in$ Id$(R)$ and $f\in$ Id$(T)$ and $E = (e, f )\in$ Id$(D(R,
T))$. Then $R$ is Zhou right $e$-reduced and $T$ is Zhou right
$f$-reduced if and only if $D(R, T)$ is Zhou right $E$-reduced.
\end{thm}
\begin{proof} For the necessity, let $(a, b)\in N(D(R, T))$. By Lemma \ref{IR} (5), $a\in N(R)$ and $a + b\in N(T)$.
Since $R$ and $T$ are Zhou right $e$-reduced and $f$-reduced, respectively, $ae\in \delta(R)$ and $(a + b)f\in T$. So $(a, b)(e,
f) = (ae, (a + b)f + be)\in \delta(D(R, T))$. The sufficiency is clear.
\end{proof}

We now give some examples for Dorroh extensions.

\begin{exs}\label{four} \em (1) Consider the ring $R = M_2(\mathbb Z_2)$ and
the subring without identity $$T =  \left \{0,
\begin{bmatrix}1&1\\1&1\end{bmatrix},
\begin{bmatrix}1&1\\0&0\end{bmatrix}, \begin{bmatrix}0&0\\1&1\end{bmatrix}\right\}$$ of $R$. Then $D(R, T)$
is Zhou right $E$-reduced for each $E\in$ Id$(D(R, T))$.\\
(2) Let $S = \{a, b\}$ be the semigroup satisfying the relation
$xy = x$ for $x, y\in S$. Therefore we have the multiplication
$a^2 = ab = a$, $b^2 = ba = b$. Put $T = \Bbb Z_2S = \{0, a, b, a
+ b\}$, which is a four-element semigroup ring without identity
seeing in \cite[Example 1]{Ma}. Let $D(\Bbb Z, T)$ denote the
Dorroh extension of $T$ by $\Bbb Z$. Then $\delta(T) = \{0,  a +
b\}$ and $\delta(D(\Bbb Z, T)) = \delta(\Bbb Z)\oplus T$. Then
$D(\Bbb Z, T)$ is
Zhou $e$-reduced for each $e\in$ Id$(D(\Bbb Z, T))$.\\
(3) Let $T$ denote the semigroup in (2) and consider the ring
\begin{center} $D(\Bbb Z_2, T) = \{(0, 0), (1, 0), (0, a), (0, b), (0, a+b), (1, a), (1, b), (1, a+b)\}.$
 \end{center} Then $D(\Bbb Z_2, T)$ is Zhou $e$-reduced for each $e\in$ Id$(D(\Bbb Z_2, T))$.
\end{exs}
\begin{proof} We firstly note the fact that $\delta(D(R, T)) = \delta(R)\oplus T$ by Lemma \ref{IR}.\\
(1) It is obvious by $\delta (D(R, T)) = R\oplus T$ since
$\delta(R) = R$.\\
(2) We infer from \cite[Example 9]{Ber} that the set of nilpotent
elements of $D(\Bbb Z, T)$, $N(D(\Bbb Z, T)) = \{(0,0), (0, a +
b)\}$, is the Jacobson radical $J(D(\Bbb Z, T))$ and the right
socle Soc$(D(\Bbb Z, T))$. It entails that $\delta(D(\Bbb Z, T)) =
\{(0,0), (0,a), (0,b), (0, a + b)\}$. For any idempotent $e\in$
Id$(D(\Bbb Z, T))$,
$N(D(\Bbb Z, T))e\subseteq \delta(D(\Bbb Z, T))$. Thus $D(\Bbb Z, T)$ is Zhou $e$-reduced for each $e\in$ Id$(D(R, T))$.\\
(3) Obviously, the proper ideals of $D(\Bbb Z_2, T)$ are
\begin{center}$<(0, a)> = \{(0, 0), (0, a)\}$,\\ $<(0, b)> = \{(0, 0), (0, b)\}$,\\ $<(0, a + b)> = \{(0, 0), (0, a + b)\}$,\\
$<(1, a)> = <(1, b)> = \{(0, 0), (1, a), (1, b), (0, a + b)\}$.
\end{center}
 An easy calculation reveals that $J(D(\Bbb Z_2, T)) = \{(0, 0), (0, a + b)\} =  N(D(\Bbb Z_2, T))$, and
$\delta(D(\Bbb Z_2, T)) = \{(0, 0), (0, a), (0, b), (0, a+b)\}.$
Hence $D(\Bbb Z_2, T)$ is Zhou $e$-reduced for each $e\in$
Id$(D(\Bbb Z_2, T))$ since $N(D(\Bbb Z_2, T))\subseteq
\delta(D(\Bbb Z_2, T)).$
\end{proof}

We cite a ring defined by Nicholson in \cite[Example 2.15]{Nic}
and also studied by Zhou in \cite[Example 4.3]{Zh}.

\begin{prop} Let $F$ be a field and consider the ring
\begin{center}
$R = \{(x_1, x_2, \dots x_n, x, x, \dots )\mid n\in\Bbb {N},
x_i\in M_2(F), x\in U_2(F)\}$
\end{center}
with componentwise operations. Then $R$ is Zhou $e$-reduced for
each $e\in$ Id$(R)$.
\end{prop}
\begin{proof}
 It is proved in \cite[Example 4.3]{Zh} that
\begin{center}
Soc$(R) = \{(x_1, x_2, \dots x_n, 0, 0, \dots )\mid n\in\Bbb {N},
x_i\in M_2(F)\}$, $J(U_2(F)) = e_{12}U_2(F)$ and $\delta(R) =
\{(x_1, x_2, \dots , x_n, x, x, \dots )\mid n\in\Bbb {N}, x_i\in
M_2(F), x\in J(U_2(F))\}$.
\end{center}
Obviously, $N(R)\subseteq \delta(R)$ since $J(U_2(F)) =
N(U_2(F))$. Then for any $e\in$ Id$(R)$, $N(R)e\subseteq
\delta(R)$. This completes the proof.
\end{proof}

Let $R$ be a ring, $\sigma\colon R\rightarrow R$ be a ring
homomorphism and $R[[x, \sigma]]$ denote the ring of skew formal
power series $\{\sum^{\infty}_{i=0} a_ix^i\mid a_i\in R\}$. The
addition is usual one and multiplication is defined by $xa =
\sigma(a)x$.
We say  that a ring $R$ satisfy the property (P) if the following holds:
\begin{center}
Every idempotent of $R[x]$ and $R[[x]]$ is in $R$.
\end{center}
The class of Armendariz rings was initiated by Armendariz in
\cite{Ar}. The ring $R$ is called {\it Armendariz} if whenever
polynomials $f(x)= \sum^n_{i = 0}a_ix^i$, $g(x)= \sum^m_{j =
0}b_jx^j\in R[x]$ satisfy $f(x)g(x) = 0$, then $a_i b_j = 0$ for
each $i$, $j$. By \cite[Lemma 1]{Ar}, reduced rings are Armendariz and Armendariz rings are abelian. Also,
abelian rings satisfy the property (P).

\begin{prop}
Let $R$ be a reduced ring and $\sigma \colon R\rightarrow R$ a
ring homomorphism. Then $R[[x, \sigma]]$ is Zhou $e$-reduced for
each $e\in$ Id$(R)$. In particular, $R[[x]]$ is Zhou $e$-reduced
for each $e\in$ Id$(R)$.
\end{prop}
\begin{proof}  Clear by the fact that $N(R[[x, \sigma]]) = N(R) = 0$.
\end{proof}

There are non-reduced rings $R$ and homomorphisms $\sigma \colon R\rightarrow R$ such that $R[x]$ and $R[[x]]$ have non trivial idempotents that are not included in $R$.

\begin{ex}\em Let $R = U_2(\Bbb Z_2)$ and consider  the ring $R[x]$ and the homomorphism
\begin{center} $\sigma \colon U_2(\Bbb Z_2)\rightarrow U_2(\Bbb Z_2)$ defined by
$\sigma(e_{11}a + e_{12}b + e_{22}c) = e_{11}a + e_{12}(-b) +
e_{22}c$.
\end{center} Then $R[x]$ is non-abelian, as well as non-reduced and non-Armendariz.
\end{ex}
\begin{proof} Let $A = \begin{bmatrix}1&0\\0&0\end{bmatrix} + \begin{bmatrix}0&1\\0&0\end{bmatrix}x\in R[x]$.
 Then $A^2 = A$. The ring $R[x]$ is non-abelian since
$E = \begin{bmatrix}1&0\\0&0\end{bmatrix}$ is a non-central
idempotent, therefore $R[x]$ is not Armendariz by \cite[Lemma
7]{KL1}. On the other hand, $N =
\begin{bmatrix}0&1\\0&0\end{bmatrix}x\in R[x]$ being a non-zero nilpotent implies
that $R[x]$ is non-reduced.
\end{proof}

One may ask that $f(x) = \sum_ia_ix^i \in N(R[x])$ if and only if
all $a_i\in N(R)$. It is positive for polynomials over commutative
rings but need not hold skew polynomial rings even over
commutative rings.

\begin{ex}\label{son} \cite{RP} \em Let $D$ be an integral domain, $R = D\times D$ and  $\sigma$ be the automorphism of $R$ switching components. Then $R$ is commutative.
Let $f(x) = (1, 0)x + (1, -1)x^2 + (0, -1)x^3\in N(R[x, \sigma])$.
Then $f(x)^2 = 0$ but the coefficient of $x^2$ is not nilpotent in
$R$.
\end{ex}


\begin{thm}
If $R$ is a commutative ring, then $R[x]$ is Zhou e-reduced for each $e\in$ Id$(R[x])$.
\end{thm}
\begin{proof} Let $R$ be a commutative ring. Then $f(x)\in N(R[x])$ if and only if the coefficients of $f(x)$ are nilpotent.
The ring $R$ being commutative implies $N(R)\subseteq \delta(R)$
and so
$N(R[x]) = N(R)[x]\subseteq \delta(R)[x]\subseteq \delta(R[x])$.
\end{proof}

\begin{ex}\rm Let  $R = \Bbb Z_2\oplus \Bbb Z_2$ and  $\sigma$ be the automorphism of $R$ switching components.
We claim  $N(R[x, \sigma])\subseteq \delta(R[x, \sigma])$, and so
$R[x, \sigma]$ is Zhou e-reduced for $e = (0, 1)$ or $e = (1, 0)$
in Id$(R[x, \sigma])$. Indeed,  note that $((1, 0)x)^2 = 0$, $((0,
1)x)^2 = 0$, $(1, 1)x(0, 1) = (1, 0)x$ and $(1, 1)x(1, 0) = (0,
1)x$. Hence $f(x) = a_0 + a_1x + a_2x^2 + \dots + a_nx^n\in N(R[x,
\sigma])$ if and only if $a_0 = 0$, and $a_i = (1, 0)$ or $a_i =
(0, 1)$ or $(1, 1)$. Let $0\neq f(x)\in N(R[x, \sigma])$. Then
$f(x)eR[x, \sigma]$ is a nil right ideal of $R[x, \sigma]$. It
entails that $f(x)e\in \delta(R[x, \sigma])$. So $N(R[x,
\sigma])e\subseteq \delta(R[x, \sigma])$. Thus $R[x, \sigma]$ is
Zhou $e$-reduced for $e = (0, 1)$ or $e = (1, 0)$ in Id$(R[x,
\sigma])$ by Proposition \ref{lem}(3).
\end{ex}
\section{Some Zhou $e$-reduced subrings of matrix rings}

In this section, we focus on some certain matrix rings in terms of  the Zhou $e$-reducedness.

{\bf The rings $H_3(\mathbb Z, R)$:} Let $R$ be a ring and
consider the ring $$H_3(\mathbb Z, R)=\left\{
\begin{bmatrix}n & a_1 & a_2 \\0 & a_3 & a_4 \\0 & 0 & n \\\end{bmatrix}\mid a_1,a_2,a_3,a_4\in R, n\in\Bbb Z
\right\}$$ with the usual matrix addition and multiplication. We
have the following.
\begin{lem}\label{son} Let $R$ be a ring. Then the following hold for the Zhou radical.\\
\item[(1)] $ N(H_3(\mathbb Z, R)) = \left\{\begin{bmatrix}0&a&b\\0&c&d\\0&0&0\end{bmatrix}\in H_3(\mathbb Z, R) \mid a,b,d\in R, c\in N(R)\right\}$.\\
\item[(2)] If $R$ is a simple ring, then $\delta(H_3(\mathbb Z, R))
= \begin{bmatrix}0&R&R\\0&0&R\\0&0&0\end{bmatrix}$, otherwise
$$\delta(H_3(\mathbb Z, R))=
\begin{bmatrix}0&R&R\\0&\delta(R)&R\\0&0&0\end{bmatrix}.$$
\end{lem}
\begin{proof} (1) One way is clear.
Let $A = \begin{bmatrix}0&a&b\\0&c&d\\0&0&0\end{bmatrix}\in H_3(\mathbb Z, R)$.
Assume that $c\in N(R)$ with $c^k = 0$ for some positive integer $k$. Then $A^{(k+2)} = 0$. Hence $A\in N(H_3(\mathbb Z, R))$.\\
(2) Let $R$ be a ring and $Emri$ denote the set of all essential
maximal right ideals in $R$. Consider the right ideals $I_1 =
\left\{\begin{bmatrix}n&a&b\\0&0&0\\0&0&n\end{bmatrix}\mid n\in
\mathbb Z, a, b\in R\right\}$, $I_2 =
\begin{bmatrix}0&0&0\\0&R&R\\0&0&0\end{bmatrix}$ of $H_3(\mathbb
Z, R)$. Then $H_3(\mathbb Z, R) = I_1\oplus I_2$.  So $\delta(I_1)
=
\begin{bmatrix}0&R&R\\0&0&0\\0&0&0\end{bmatrix}$. If $R$ is
simple, then $I_2$ is local and $\delta(I_2) =
\begin{bmatrix}0&0&0\\0&0&R\\0&0&0\end{bmatrix}$. Assume that $R$
is not simple. Then $\bigcap\limits_{I\in
Emri}\begin{bmatrix}0&0&0\\0&I&R\\0&0&0\end{bmatrix}$ is
$\delta(I_2)$. So $\delta(I_2) =
\begin{bmatrix}0&0&0\\0&\delta(R)&R\\0&0&0\end{bmatrix}$. It
follows that $\delta(H_3(\mathbb Z, R)) = \delta(I_1) \oplus
\delta(I_2)$. It depends on the simplicity of $R$.
\end{proof}

The following example shows the Zhou radicals related to the rings
$H_3(\mathbb Z, R)$.
\begin{ex}\label{zrad}\rm $\delta(H_3(\mathbb Z, \mathbb Z_2)) = \begin{bmatrix}0&\mathbb Z_2&\mathbb Z_2\\0&0&\mathbb Z_2\\0&0&0\end{bmatrix}$ and $\delta(H_3(\mathbb Z, \mathbb Z_4)) = \begin{bmatrix}0&\mathbb Z_4&\mathbb Z_4\\0&2\mathbb Z_4&\mathbb Z_4\\0&0&0\end{bmatrix}$.\\
\end{ex}

\begin{thm} The following hold.
\begin{enumerate}
\item[{(1)}] There are simple rings $R$ such that $H_3(\mathbb Z, R)$ are Zhou right (left) $E$-reduced for each $E\in$ Id$(H_3(\mathbb Z, R))$.
\item[{(2)}] There are simple rings $R$ such that $H_3(\mathbb Z, R)$
need not be Zhou right $E$-reduced for some $E\in$
Id$(H_3(\mathbb Z, R))$.
\item[{(3)}] There are rings $R$ that are not simple such that
$H_3(\mathbb Z, R)$ are Zhou right (left) $E$-reduced for some
$E\in$ Id$(H_3(\mathbb Z, R))$. \end{enumerate}
\end{thm}
\begin{proof}
(1) Let  $A = \begin{bmatrix}n&a&b\\0&c&d\\0&0&n\end{bmatrix}\in N(H_3(\mathbb Z, \mathbb Z_2))$. Then $c = 0$, $n = 0$. Hence $AE, EA\in \delta(H_3(\mathbb Z, \mathbb Z_2))$ for each $E\in$ Id$(H_3(\mathbb Z, \mathbb Z_2))$.\\
(2) Let $X = \begin{bmatrix}1&1\\0&0\end{bmatrix}$, $Y =
\begin{bmatrix}0&0\\1&0\end{bmatrix}$, $Z =
\begin{bmatrix}0&0\\0&1\end{bmatrix}$, $T = 0\in M_2(\mathbb
Z_2)$, $A = \begin{bmatrix}0&X&T\\0&Y&Z\\0&0&0\end{bmatrix}\in
N(H_3(\mathbb Z, M_2(\mathbb Z_2)))$ and $E =
\begin{bmatrix}0&0&0\\0&I_2&I_2\\0&0&0\end{bmatrix}\in$
Id$(H_3(\mathbb Z, M_2(\mathbb Z_2)))$. Note that
$AE = \begin{bmatrix}0&X&X\\0&Y&Y\\0&0&0\end{bmatrix}\notin \delta(H_3(\mathbb Z, M_2(\mathbb Z_2)))$ and $EA = \begin{bmatrix}0&0&0\\0&Y&Z\\0&0&0\end{bmatrix}\notin \delta(H_3(\mathbb Z, M_2(\mathbb Z_2)))$.\\
(3) Let $A = \begin{bmatrix}0&a&b\\0&c&d\\0&0&0\end{bmatrix}\in
N(H_3(\mathbb Z, \mathbb Z_4))$. There are two possibilities for
$c$, that is, $c = 0$ or $c = 2$. In case $c = 0$, there is nothing to do
since $A\in \delta(H_3(\mathbb Z, \mathbb Z_4))$. Otherwise, let $c
= 2$ and $E =
\begin{bmatrix}0&1&0\\0&1&0\\0&0&0\end{bmatrix}$. Then $AE =
\begin{bmatrix}0&a&0\\0&2&0\\0&0&0\end{bmatrix}\in
\delta(H_3(\mathbb Z, \mathbb Z_4))$ and $EA =
\begin{bmatrix}0&2&d\\0&2&d\\0&0&0\end{bmatrix}\in
\delta(H_3(\mathbb Z, \mathbb Z_4))$. In fact, $H_3(\mathbb Z,
\mathbb Z_4)$ is Zhou right and left $E$-reduced for some $E\in$
Id$(H_3(\mathbb Z, R))$.
\end{proof}
{\bf The rings $H_{(s,t)}(R)$:} Let $R$ be a ring and  $s$, $t\in
C(R)$ be invertible in $R$. Let\begin{center} $H_{(s,t)}(R) = \left
\{\begin{bmatrix}a&0&0\\c&d&f\\0&0&g
\end{bmatrix}\in M_3(R)\mid a, c, d, f, g\in R, a - d = sc, d - g = tf\right \}$.\end{center}
Then $H_{(s,t)}(R)$ is a subring of $M_3(R)$.

\begin{lem}\label{son1} Let  $A = \begin{bmatrix}a&0&0\\c&d&f\\0&0&g\end{bmatrix}\in H_{(s,t)}(R)$. Then
\begin{enumerate}
\item[(1)] $A\in N(H_{(s,t)}(R))$ if and only if $a$, $d$, $g\in N(R)$.
\item[(2)] $A\in \delta(H_{(1,1)}(R))$ if and only if $a$, $d$, $g\in \delta(R)$.
\item[(3)] $A\in$ Id$(H_{(1,1)}(R))$ if and only if $a$, $d$, $g\in$ Id$(R)$.
\end{enumerate}
\end{lem}
\begin{proof} (1) One way is clear. For the other way, suppose that $a$, $d$, $g\in N(R)$ with $n$ nilpotency index of triples $a$, $d$, $g$, so that $a^n = 0$, $d^n = 0$ and $g^n = 0$. Then $A^n = \begin{bmatrix}a^n&0&0\\**&d^n&*\\0&0&g^n\end{bmatrix} = \begin{bmatrix}0&0&0\\**&0&*\\0&0&0\end{bmatrix}$. So $A^{2n} = 0$.\\
(2) Let $e_{ij}$ denote the matrix unit having $1$ at the  $(i, j)$ entry and
$0$ elsewhere, and consider $I_1 = \begin{bmatrix}1&0&0\\1&0&0\\0&0&0\end{bmatrix}R$,
$I_2 = \begin{bmatrix}0&0&0\\-1&1&1\\0&0&0\end{bmatrix}R$ and
$I_3 = \begin{bmatrix}0&0&0\\0&0&1\\0&0&-1\end{bmatrix}R$. Then $I_1$, $I_2$ and $I_3$ are right ideals of
$H_{(1,1)}(R)$ and $H_{(1,1)}(R) = I_1\oplus I_2\oplus I_3$. Hence
$\delta(H_{(1,1)}(R)) = \delta(I_1) \oplus \delta(I_2)\oplus \delta(I_3)$. Thus
$\delta(I_1) = \begin{bmatrix}1&0&0\\1&0&0\\0&0&0\end{bmatrix}\delta(R)$,
$\delta(I_2) =  \begin{bmatrix}0&0&0\\-1&1&1\\0&0&0\end{bmatrix}\delta(R)$
and $\delta(I_3) = \begin{bmatrix}0&0&0\\0&0&1\\0&0&-1\end{bmatrix}\delta(R)$.
Therefore $$\delta(H_{(1,1)}(R)) = \left\{\begin{bmatrix}a&0&0\\a-d&d&d-g\\0&0&g\end{bmatrix}\mid a, d, g\in\delta(R)\right\}.$$\\
(3) It is clear that $A^2 = A$ implies $a$, $d$, $g\in$ Id$(R)$.
Conversely, assume that $A\in H_{(1,1)}(R)$ with $a$, $d$, $g\in$
Id$(R)$. Having $a-d = c$ and $d-g=f$ imply, in turn, $c = a - d
= a^2 - da + da - d^2 = (a-d)a + d(a-d) = ca + dc$ and $f = d-g =
d^2 - dg + dg - g^2 = d(d - g) + (d - g)g = df + fg$. They entail
that $A^2 = A$.
\end{proof}
\begin{thm} A ring $R$ is Zhou right $e$-reduced for each $e\in$Id$(R)$ if and only if $H_{(1,1)}(R)$ is Zhou right $E$-reduced for each $E\in$ Id$(H_{(1,1)}(R))$.
\end{thm}
\begin{proof} For the forward direction, let $A = \begin{bmatrix}a&0&0\\c&d&f\\0&0&g \end{bmatrix}\in N(H_{(1,1)}(R))$ and $E = \begin{bmatrix}x&0&0\\y&z&t\\0&0&u\end{bmatrix}\in$ Id$(H_{(1,1)}(R))$. Then $a$, $d$, $g\in N(R)$ and $x$, $z$, $u\in$ Id$(R)$. Since $R$ is Zhou right $e$-reduced, $ax$, $dz$, $gu\in \delta(R)$. By Lemma \ref{son1}(2), $AE\in \delta(H_{(1,1)}(R))$.\\
For the backward direction, let $a\in N(R)$ and $e\in$ Id$(R)$. Then $A = aI_3\in N(H_{(1,1)}(R))$ and $E = eI_3\in$ Id$(H_{(1,1)}(R))$. Since $H_{(1,1)}(R)$ is Zhou right $E$-reduced, $AE\in \delta(H_{(1,1)}(R))$. By Lemma \ref{son1}(2), $ae\in \delta(R)$.
\end{proof}

{\bf Generalized matrix rings:} Let $R$ be a ring and $s$ a
central element of $R$. Then  $\begin{bmatrix}
R&R\\R&R\end{bmatrix}$ becomes a ring denoted by $K_s(R)$ with
addition defined componentwise and multiplication defined in
\cite{Kr} by
$$\begin{bmatrix} a_1&x_1\\y_1&b_1\end{bmatrix}\begin{bmatrix}
a_2&x_2\\y_2&b_2\end{bmatrix} = \begin{bmatrix}a_1
a_2+sx_1y_2&a_1x_2+x_1b_2\\y_1a_2+b_1y_2&sy_1x_2+b_1b_2\end{bmatrix}.$$
The ring $K_s(R)$ is called a {\it generalized matrix ring
over $R$}.
\begin{lem}\label{delt} Let $R$ be a ring and $A = \begin{bmatrix} a&x\\y&b\end{bmatrix}\in K_0(R)$. Then the following hold.
\begin{enumerate}
\item $A \in N(K_0(R))$ if and only if $a$, $b\in N(R)$.
\item $A\in \delta(K_0(R))$ if and only if $a$, $b\in \delta(R)$.
\item If $A\in$ Id$(K_0(R)))$, then $a$, $b\in$ Id$(R)$.
\end{enumerate}
\end{lem}
\begin{proof} (1) Let $A = \begin{bmatrix} a&x\\y&b\end{bmatrix}\in N(K_0(R))$
with $A^n = 0$ for some positive integer $n$. Then $a^n = 0$ and $b^n = 0$. Conversely,
let $a^n = 0$, $b^m = 0$ and $k = max\{n,m\}$. Then $A^k = \begin{bmatrix} 0&*\\**&0\end{bmatrix}$.
 Hence $A^{2k} = 0$.\\
(2) Let $\delta(K_0(R)) = \begin{bmatrix}K&M\\N&L\end{bmatrix}$ and $e_{11}\in$ Id$(K_0(R))$.
By Lemma \ref{mod}(1), we have the relations $K = e_{11}\delta(K_0(R))e_{11} = \delta(e_{11}K_0(R)e_{11})
= \delta(R)$. Similarly, we may get $L = \delta(R)$. Also, $X =
\begin{bmatrix}0&M\\0&0\end{bmatrix}$ and $Y = \begin{bmatrix}0&0\\N&0\end{bmatrix}$
are nilpotent right ideals of $K_0(R)$. It entails that $X$ and $Y$ are contained in $\delta(K_0(R))$.\\
(3) Let $A = \begin{bmatrix}a&x\\y&b\end{bmatrix}\in$
Id$(K_0(R))$. Then $a^2 = a$, $b^2 = b$.
\end{proof}

The converse statement of Lemma \ref{delt}(3) need not hold in general.
\begin{ex} {\em Let $A = \begin{bmatrix}1&0\\1&1\end{bmatrix}\in K_0(\mathbb Z_7)$.
Then diagonal entries of $A$ are idempotent but $A$ is not
idempotent. In fact, $A^2 =
\begin{bmatrix}1&0\\2&1\end{bmatrix}\neq A$. Hence $A\not\in$
Id$(K_0(\mathbb Z_7))$.}
\end{ex}
\begin{thm} A ring $R$ is Zhou right $e$-reduced for each $e\in$ Id$(R)$ if and only if $K_0(R)$ is Zhou right $E$-reduced for each $E\in$ Id$(K_0(R))$.
\end{thm}
\begin{proof} For the necessity, let $A = \begin{bmatrix} a&x\\y&b\end{bmatrix}\in N(K_0(R))$ with $A^n = 0$ for some positive integer $n$.
 By Lemma \ref{delt}(1), $a$, $b\in N(R)$. Let $E = \begin{bmatrix} e&*\\ *&f\end{bmatrix}\in$ Id$(K_0(R))$. Then $e$, $f\in$ Id$(R)$. Since $R$ is Zhou right $e$-reduced for each $e\in$ Id$(R)$, $ae$, $bf\in \delta(R)$. By Lemma \ref{delt}(2), $AE = \begin{bmatrix} ae&**\\ **&bf\end{bmatrix}\in \delta(K_0(R))$.
For the sufficiency, let $a\in N(R)$, $e\in$ Id$(R)$, $A =
\begin{bmatrix} a&0\\0&a\end{bmatrix}$ and $E = eI_2$. Then $A\in
N(K_0(R))$, $E\in$ Id$(K_0(R))$ and $AE = \begin{bmatrix}
ae&0\\0&ae\end{bmatrix}$. Since $K_0(R)$ is Zhou right
$E$-reduced, $AE\in \delta(K_0(R))$. By Lemma \ref{delt}(2),
$ae\in \delta(R)$.
\end{proof}

\noindent{\bf Disclosure statement.} The authors report there are
no competing interests to declare.

\end{document}